\documentclass[12pt,reqno]{article}

\usepackage[usenames]{color}
\usepackage{amssymb}
\usepackage{amsmath}
\usepackage{amsthm}
\usepackage{amsfonts}
\usepackage{amscd}
\usepackage{graphicx}

\usepackage[colorlinks=true,
linkcolor=webgreen,
filecolor=webbrown,
citecolor=webgreen]{hyperref}

\definecolor{webgreen}{rgb}{0,.5,0}
\definecolor{webbrown}{rgb}{.6,0,0}

\usepackage{color}
\usepackage{fullpage}
\usepackage{float}

\usepackage{graphics}
\usepackage{latexsym}
\usepackage{epsf}
\usepackage{breakurl}

\newcommand{\seqnum}[1]{\href{https://oeis.org/#1}{\rm \underline{#1}}}

\begin{document}

\theoremstyle{plain}
\newtheorem{theorem}{Theorem}
\newtheorem{corollary}[theorem]{Corollary}
\newtheorem{lemma}[theorem]{Lemma}
\newtheorem{proposition}[theorem]{Proposition}

\theoremstyle{definition}
\newtheorem{definition}[theorem]{Definition}
\newtheorem{example}[theorem]{Example}
\newtheorem{conjecture}[theorem]{Conjecture}
\newtheorem{question}[theorem]{Question}

\theoremstyle{remark}
\newtheorem{remark}[theorem]{Remark}

\begin{center}
\vskip 1cm{\LARGE\bf 
Counting Bubbles in Linear Chord Diagrams
}
\vskip 1cm 
\large
Donovan Young\\
St Albans, Hertfordshire\\
United Kingdom\\
\href{mailto:donovan.m.young@gmail.com}{\tt donovan.m.young@gmail.com} \\
\end{center}

\vskip .2 in

\begin{abstract}
In a linear chord diagram a short chord is one which joins adjacent
vertices. We define a bubble to be a region in a linear chord diagram
devoid of short chords. We derive a formal generating function
counting bubbles by their size and find an exact result for the mean
bubble size. We find that once one discards diagrams which have no
short chords at all, the distribution of bubble sizes is given by a
smooth function in the limit of long diagrams. Using a summation over
short chords, the exact form of this asymptotic distribution is found.
\end{abstract}

\section{Introduction and basic notions}

A linear chord diagram consists of a linear arrangement of $2n$
vertices. Each vertex is joined to exactly one different vertex by an
unoriented arc called a {\it chord}. Hence every linear
chord diagram on $2n$ vertices has exactly $n$ chords. As the chords
are distinguished only by the positions of their endpoints, it is
evident that there are $(2n-1)(2n-3)\cdots 1 = (2n-1)!!$ different
linear chord diagrams on $2n$ vertices.

One interesting way of refining this counting\footnote{The
  combinatorics of linear chord diagrams has a long history beginning
  with Touchard \cite{T} and Riordan's \cite{R} studies of the number of
  chord crossings; cf.\ Pilaud and Ru\'{e} \cite{PR} for a modern
  approach and further developments. Krasko and Omelchenko \cite{KO}
  provide a more complete list of references.} is by the number of
so-called {\it short} chords, i.e., chords which join adjacent
vertices. Kreweras and Poupard~\cite{KP} provided recurrence relations
and closed form expressions for the number of diagrams with exactly
$\ell$ short chords. They also showed that the mean number of short
chords is $1$, which implies that the total number of short chords is
equinumerous with the total number of linear chord diagrams,
cf.\ \cite{CK}. Kreweras and Poupard~\cite{KP} showed further that all
higher factorial moments of the distribution approach $1$ in the
$n\to\infty$ limit, thus establishing the Poisson nature of the
asymptotic distribution.

\begin{figure}[H]
\begin{center}
\includegraphics[bb=0 0 320 56, width=3.5in]{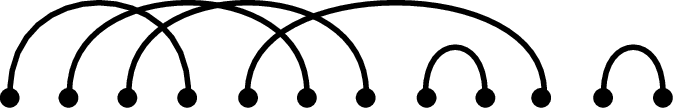}
\end{center}
\caption{A linear chord diagram on 12 vertices consisting of one
  bubble of size 1 (bounded by the two short chords) and another bubble of
  size 7 (bounded by the start of the diagram and one of the short
  chords).}
\label{f}
\end{figure}

In this paper we will be concerned with counting certain sets of
adjacent vertices of a linear chord diagram. We call these sets {\it
  bubbles}. The vertices of a given bubble may be joined by chords to
one another or to vertices outside, but (in either case) never via a
short chord. A bubble is therefore bounded either by short chords or
by the ends of the diagram, see Figure \ref{f}. The {\it size} of a
bubble is defined as the number of vertices it has, and may
generically take values from 1 to $2n-2$, or, in the case of a linear
chord diagram devoid of short chords, $2n$.

Let the total number of bubbles of size $p$ found among linear chord
diagrams on $2n$ vertices be given by $B_{n,p}$. Table \ref{TabBnp} shows
the values of $B_{n,p}$ for $n\leq 6$.

\begin{table}[H]
\begin{center}
  \begin{tabular}{c|llllllllllll}
  $n$ \textbackslash$p$& 1& 2& 3& 4& 5& 6& 7& 8& 9& 10& 11& 12\\
  \hline  
1& 0& 0\\
2& 2& 0& 0& 1\\
3& 8& 4& 2& 2& 0& 5\\
4& 42& 30& 20& 15& 12& 10& 0& 36\\
5& 300& 240& 186& 147& 120& 99& 82& 72& 0& 329\\
6& 2730& 2310& 1920& 1605& 1356& 1155& 988& 848& 730& 658& 0& 3655\\
\end{tabular}
\end{center}
\caption{Total number $B_{n,p}$ of bubbles of size $p$ counted across
  all possible linear chord diagrams with $n$ chords, On-line
  Encyclopedia of Integer Sequences \seqnum{A367000}. The last entries
  in each row are given by \seqnum{A278990}.}
\label{TabBnp}
\end{table}

\noindent There are several interesting patterns to note in Table
\ref{TabBnp}. First of all, we note that since a single diagram can
consist of many bubbles, the row-sums are generically greater than the
total number, $(2n-1)!!$, of linear chord diagrams. Secondly, the
final entries in each row are the number of configurations consisting
of a single bubble whose size is the length of the diagram; these are
therefore linear chord diagrams devoid of short chords. These
configurations have been studied elsewhere, for example by Kreweras
and Poupard~\cite{KP}, and the On-line Encyclopedia of Integer
Sequences entry \seqnum{A278990} gives these numbers as part of the
greater sequence $d_{n,s}$ (\seqnum{A079267}) of linear chord diagrams
refined by the number $s$ of short chords. We may state therefore that
$B_{n,2n}=d_{n,0}$. Thirdly, a bubble of size $2n-1$ can not be formed
(as a short chord occupies two vertices), therefore the penultimate
entry in each row is naturally zero: $B_{n,2n-1}=0$. Fourthly, the
third-from-last entry in each row is double the last entry of the row
above. This is because a bubble of size $2n-2$ can be formed in only
two ways: by placing a short chord at either end of the diagram, thus
producing a linear chord diagram devoid of short chords on two fewer
vertices. Hence $B_{n,2n-2}=2B_{n-1,2n-2}$. One may try to continue
this logic, for example, to the fourth-from-last entry. Bubbles of
size $2n-3$ are formed by a short chord positioned one vertex away
from either end of the diagram. There is therefore a chord which
connects the first (or last) vertex of the diagram with a vertex
within the bubble. We may therefore construct the bubble by starting
with configurations devoid of short chords of length $2n-4$ and
inserting this vertex in one of the $2n-3$ gaps between existing
vertices, or start with configurations of length $2n-4$ containing a
single short chord and place this vertex within it. We therefore have
the following relation:
\begin{equation}\nonumber
B_{n,2n-3} = 2 \left( (2n-3)B_{n-2,2n-4} + d_{n-2,1} \right).
\end{equation}
There will be similar relations as we move further in to each row of
the table; they will, however, become increasingly complex.

\section{Enumeration of bubbles}
\subsection{Matching polynomial method}
The present author \cite{DY3} developed a technique for computing
generating functions which count linear chord diagrams refined by
short chords. As we will use this method in the following subsection
to enumerate bubbles, we give an account of the method here first. The
method centres upon the matching (or rook) polynomial for the path of
length $2n$. We remind the reader that the matching polynomial
$m_G(z)=\sum\rho_j z^{j}$ of a graph $G$ has coefficients $\rho_j$
which count the number of $j$-edge matchings on $G$. By convention the
number $\rho_0$ of zero-edge matchings is defined to be $1$ for every
graph. For example, the path consisting of the four vertices $A$, $B$,
$C$, and $D$, has three $1$-edge matchings: $(AB)$, $(BC)$, and
$(CD)$. It has one $2$-edge matching: $(AB)(CD)$. The matching
polynomial for the path of length four is therefore $1+3z+z^2$.

The number $d_{n,0}$ of linear chord diagrams on $2n$ vertices devoid of short
chords may be calculated from the $\rho_j$ associated with the path of
length $2n$ via inclusion-exclusion:
\begin{equation}\label{incexc}
  d_{n,0}=\sum_{j=0}^{n} (-1)^j(2n-2j-1)!! \,\rho_j.
\end{equation}
The explanation is as follows. We note that for each of the $\rho_j$
choices of $j$ edges on which to place $j$ short chords, there remains
$(2n-2j-1)!!$ configurations on the remaining $2n-2j$ vertices. There
will be some number of configurations with exactly $q$ short chords
among these $(2n-2j-1)!!$. Then $(2n-2j-1)!!\,\rho_j$ counts the
number $d_{n,q+j}$ of $(q+j)$-short-chord configurations ${q+j\choose
  j}$ times. We thus have
\begin{equation}\nonumber
\sum_{j=0}^{n} (-1)^j(2n-2j-1)!! \,\rho_j =\sum_{j=0}^{n} (-1)^j
\sum_{q=0}^{n-j} {q+j\choose j} d_{n,q+j}$$
$$=d_{n,0} + \sum_{q+j=1}^n d_{n,q+j} \sum_{j=0}^{q+j} (-1)^j
    {q+j\choose j},
\end{equation}
and so all but the $0$-short-chord configurations cancel.

The alternating sum in Equation (\ref{incexc}) may be repackaged as
an integral involving the matching polynomial. It is more convenient to
collect the matching polynomials for all paths together into a
two-variable generating function, cf.\ \cite[Proposition 15]{DY3}:
\begin{equation}\nonumber
L(x,y) = \frac{1}{1-y(1+x^2y)},
\end{equation}
where $[y^{2n}]L(x,y)$ is the matching polynomial $m(x^2)$ for the
path consisting of $2n$ vertices; the replacement $z\to x^2$ is a
useful redefinition for the calculations which follow.

The generating function for the $d_{n,0}$ is then given by
\begin{equation}\nonumber
\sum_n d_{n,0}\, y^{2n}=
  \int_0^\infty dt \,e^{-t}\,\frac{1}{2\pi i}
    \oint_{|x|=\epsilon} \frac{dx}{x}\, e^{x^2/2} \,
         L\left(ix/t,yt/x \right).
\end{equation}
The explanation of this equality is as follows. A generic term in the
expansion of $L\left(ix/t,yt/x \right)$ will have the form
\begin{equation}\nonumber
\left(-\frac{x^2}{t^2}\right)^{j}\left(\frac{y^2t^2}{x^2}\right)^{n}\rho_j,
\end{equation}
where odd powers of $y$, though present in $L(x,y)$, do not survive
the contour integration over $x$, and have thus been omitted. The only
term in the expansion of the exponential $\exp(x^2/2)$ surviving the
contour integration will be $x^{2(n-j)}/(2^{n-j}(n-j)!)$, as this will
absorb the $j-n$ powers of $x^2$ in the expansion of $L\left(ix/t,yt/x
\right)$. Finally, the integration over $t$ provides a factor of
$(2n-2j)!$. We therefore have that
\begin{equation}\nonumber
  \int_0^\infty dt \,e^{-t}\,\frac{1}{2\pi i}
    \oint_{|x|=\epsilon} \frac{dx}{x}\, e^{x^2/2} \,
    L\left(ix/t,yt/x \right) =
    \sum_{j=0}^n (-1)^j\frac{(2n-2j)!}{2^{n-j}(n-j)!}\,\rho_j,
\end{equation}
which is Equation (\ref{incexc}).

\subsection{Application of the method to the enumeration of bubbles}
A bubble is bounded to the left and right by short chords (or the ends
of the diagram), thus we define the following generating function:
\begin{equation}\nonumber
  {\cal L}(r,w,x,y)=\frac{1}{1-y^2L(x,ry)}
  \Bigl(L(x,wry)-1\Bigr)
    \frac{1}{1-y^2L(x,ry)}.
\end{equation}
The central factor $L(x,wry)-1$ corresponds to the bubble being
counted, whereas the two factors of $(1-y^2L(x,ry))^{-1}$ correspond
to the remainder of the diagram, to the left and to the right of the
bubble; the presence of $y^2$ corresponds to the short chords bounding
the bubble. In the expansion of ${\cal L}(r,w,x,y)$, each power of $y$
corresponds to a vertex of the linear chord diagram. Those vertices
belonging to the bubble under consideration are further labelled with
a power of $w$, while every vertex not part of a short chord is also
labelled with a power of $r$. Armed with this aggregate generating
function ${\cal L}(r,w,x,y)$ for the matching polynomial, we calculate
the generating function which counts bubbles as follows. Let $\rho_j$
be the number of $j$-edge matchings on the vertices marked with $r$
and let $2q$ be the number of these vertices, then
\begin{equation}\label{BR}
  \sum_{j=0}^q (-1)^j\,(2(q-j)-1)!!\,\rho_j =
  \sum_{j=0}^q (-1)^j\frac{(2(q-j))!}{2^{q-j}(q-j)!}\,\rho_j
\end{equation}
counts the number of configurations with no short chords on the $2q$
vertices in question.

\begin{theorem}\label{ThmInt}
  The generating function which counts the numbers $B_{n,p}$ is given by
\begin{equation}\nonumber
B(y^2,w)=\sum_{n,p} B_{n,p}\, y^{2n} w^p =
  \int_0^\infty dt \,e^{-t}\,\frac{1}{2\pi i}
    \oint_{|x|=\epsilon} \frac{dx}{x}\, e^{x^2/2} \,
    {\cal L}\left(t/x,w,ix/t,y \right).
\end{equation}
\end{theorem}
\begin{proof}
A generic term in the expansion of ${\cal L}\left(t/x,w,ix/t,y
\right)$ (which survives the contour integration) will have the form
  \begin{equation}\nonumber
    \left(\frac{t^2}{x^2}\right)^{q}w^p
    \left(-\frac{x^2}{t^2}\right)^{j} y^{2n} \rho_j,
  \end{equation}
  where $\rho_j$ are the aforementioned matching numbers. The third
  factor follows from the fact that $L(x,y)$ is an even function of
  $x$.  The expansion of the exponential in $x^2$ will contribute only
  the term of order $x^{2(q-j)}$, as the contour integration in $x$
  will eliminate all other terms. This mechanism also forces the power
  of $r$ to be even, and accounts for the even power of $t/x$ in the
  first factor. We note that
\begin{equation}\nonumber
[x^{2(q-j)}]e^{x^2/2} = \frac{1}{2^{q-j} (q-j)!},    
\end{equation}
and that
\begin{equation}\nonumber
\int_0^\infty dt\, e^{-t} \,\frac{(t^2)^q}{(-t^2)^j} = (-1)^j\,(2(q-j))!,
\end{equation}
and so together they give the factor $(-1)^j\,(2(q-j)-1)!!$ as required by
Equation (\ref{BR}).
\end{proof}

\begin{corollary}\label{CorGF}
The generating function $B(y,w)$ is given by
\begin{align}\nonumber
  &B(y,w)=\int_0^\infty dt \,e^{-t}\Biggl(
  \frac{(1-w)^2(1-wy)^2}{(1+w^2y)\left(1-w(1-wy)\right)^2}
\exp \frac{y}{2}\left(\frac{tw}{1+w^2y}\right)^2\\\nonumber
&+\left(\frac{2 - w(y + 2)\left(1+y(1-w(2-wy))\right)}
{\left(1-w(1-wy)\right)^2}wy
+\frac{1-wy}{1-w(1-wy)} t^2wy^3 \right)\exp\frac{t^2y}{2}
\Biggr)
\end{align}
\end{corollary}
\begin{proof}
This is established by evaluating the contour integral over $x$ from
Theorem \ref{ThmInt}. We begin with a simplified expression for ${\cal
  L}\left(t/x,w,ix/t,y\right)$:
\begin{equation}\nonumber
{\cal L}\left(t/x,w,ix/t,y\right)= \frac{w y (t-w x y) (x-t y+x
  y^2)^2}{(1+w^2y^2)} \frac{1}{(x-t y)^2 \left(x-\frac{t w
    y}{1+w^2y^2}\right)}.
\end{equation}
We therefore have two poles: a simple pole at $x=twy(1+w^2y^2)^{-1}$,
and a pole of order two at $x=ty$. Computing the residues at these
poles we find, through direct calculation, the result for $B(y,w)$.
\end{proof}

\begin{lemma}\label{LemTot}
  The total number of bubbles, counted across all linear chord
  diagrams on $2n$ vertices is given by
\begin{equation}\nonumber
  \sum_{p=1}^{2n} B_{n,p} = 
  \frac{(2n-2)!(4n-5)}{2^{n-1}(n-1)!}.
\end{equation}
\end{lemma}
\begin{proof}
  We use Corollary \ref{CorGF} and set $w=1$:
  \begin{align}\nonumber
    B(y,1)= \sum_n y^n \sum_{p=1}^{2n} B_{n,p} &=
    \int_0^\infty dt \,e^{-t}
    \left(1 - y - (1 - t^2)y^2 - t^2 y^3\right) \exp\frac{t^2y}{2}\\\nonumber
    &= \sum_{j}
    (2j-1)!!y^{j} \left( 1-y-y^2 +2j(2j+1)y^2(1-y)\right),
  \end{align}
  where in the second line we have used the expansion of the
  exponential. Our result is then obtained by reading off the
  coefficient of $y^n$:
  \begin{equation}\nonumber
[y^n]\sum_{j}
  (2j-1)!!y^{j} \left( 1-y-y^2 +2j(2j+1)y^2(1-y)\right)
  =\frac{(2n-2)!(4n-5)}{2^{n-1}(n-1)!}.
\end{equation}
\end{proof}

\begin{lemma}\label{LemFrst}
  The un-normalized first moment of bubble size is given by
\begin{equation}\nonumber
  \sum_{p=1}^{2n} p\,B_{n,p} = \frac{(2n-1)!}{2^{n-2}(n-2)!}.
\end{equation}
\end{lemma}
\begin{proof}
We use Corollary \ref{CorGF} and take a derivative with respect to $w$:  
\begin{align}\nonumber
    \frac{\partial}{\partial w}B(y,1)&= \sum_n y^n \sum_{p=1}^{2n} p \,B_{n,p} =
    \int_0^\infty dt \,e^{-t}
    \left(t^2y(1 - 2y) - 2y\right) \exp\frac{t^2y}{2}\\\nonumber
    &=\sum_{j}
  (2j-1)!!y^{j} \left(-2y +2j(2j+1)y(1-2y)\right),
\end{align}
where we have used the expansion of the exponential before integrating
over $t$. Our result is then obtained by reading off the coefficient of $y^n$:
  \begin{equation}\nonumber
    [y^n]\sum_{j}
  (2j-1)!!y^{j} \left(-2y +2j(2j+1)y(1-2y)\right)
=\frac{(2n-1)!}{2^{n-2}(n-2)!}.
\end{equation}
\end{proof}

\begin{theorem}\label{Thmmean}
  The mean bubble size $\bar p$, taken over all bubbles on linear
  chord diagrams consisting of $2n$ vertices, is given by
  \begin{equation}\nonumber
    \bar p \equiv \frac{\sum_{p=1}^{2n} p\,B_{n,p}}
    {\sum_{p=1}^{2n} B_{n,p}}
      = \frac{2(2n-1)(n-1)}{4n-5}.
    \end{equation}
\end{theorem}
\begin{proof}
  We employ the results of Lemmas \ref{LemTot} and \ref{LemFrst}.
  \end{proof}

\section{Asymptotic distribution}
\label{Secasymp}

The values $B_{n,p}$ for $1\leq p \leq 2n-2$ follow a smooth
distribution in the limit $n\to\infty$, the form of which we shall
discover in this section. It is therefore natural to discard the
penultimate and final entries in Table \ref{TabBnp}, and to view the
configurations consisting of a single bubble of size $2n$ as part of a
different counting problem. Before we do this, however, we note that
the result of Theorem \ref{Thmmean} implies that the asymptotic value
of the mean bubble size $\bar p$ is
\begin{equation}\nonumber
\lim_{n\to \infty}   \frac{2(2n-1)(n-1)}{4n-5} = n.
\end{equation}
To see how this arises, we remind the reader that Kreweras and
Poupard~\cite{KP} proved that the number $k$ of short chords is
asymptotically Poisson distributed with mean $1$. To leading order,
the presence of $k$ short chords induces $k+1$ bubbles. This implies
\begin{equation}\label{EqPoisOne}
  \sum_{p=1}^{2n} B_{n,p} \simeq (2n-1)!!\sum_{k\geq 0} (k+1) \frac{e^{-1}}{k!}
  =2 (2n-1)!!,
\end{equation}
where $(2n-1)!!$ counts the total number of linear chord diagrams
consisting of $n$ chords. Since the distribution of the positions of
these short chords is asymptotically uniform, the mean size of a
bubble is asymptotically $2n/(k+1)$. The overall mean bubble size is
then
\begin{equation}\label{EqPoisTwo}
  \frac{\sum_{p=1}^{2n} p\,B_{n,p}}
    {\sum_{p=1}^{2n} B_{n,p}} \simeq \frac{1}{2(2n-1)!!}
  (2n-1)!!\sum_{k\geq 0} (k+1)\frac{2n}{k+1} \frac{e^{-1}}{k!}
  =n.
\end{equation}

We now turn our attention to the asymptotic distribution of bubble
sizes, not including the values $p=2n-1, 2n$. The mean of this
distribution can be obtained by subtracting the contribution of the
bubble of size $2n$ as follows:
\begin{equation}\label{asympmean}
  \frac{\sum_{p=1}^{2n-2} p\,B_{n,p}}
    {\sum_{p=1}^{2n-2} B_{n,p}}\simeq
  \frac{\sum_{p=1}^{2n} p\,B_{n,p}-2n e^{-1}(2n-1)!!}
       {\sum_{p=1}^{2n} B_{n,p}- e^{-1}(2n-1)!!}
       \simeq \frac{2 - 2e^{-1}}{2 - e^{-1}}\,n \simeq 0.7746\,n,
\end{equation}
where we have made use of Equation (\ref{EqPoisOne}) and (\ref{EqPoisTwo}).

\begin{figure}[H]
\begin{center}
\includegraphics[bb=0 0 334 124, width=5.5in]{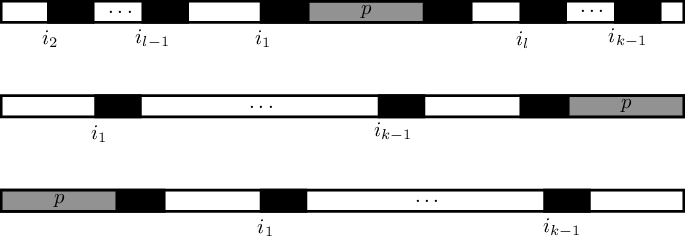}
\end{center}
\caption{Contribution of diagrams with $k$ short chords (shown in
  black) to the asymptotic enumeration of bubbles of size $p$. The
  indices $i_j$ are positions which must be summed over.}
\label{Figbubbasymp}
\end{figure}

\begin{theorem}\label{ThmMain}
Let $p$ denote the size of a bubble and let $x=p/(n-1)$. The
asymptotic distribution $\rho(x)$ of the size of bubbles, excluding
diagrams which are themselves bubbles, is given for $x\in (0,2]$ by
\begin{equation}\nonumber
  \rho(x) = \lim_{n\to\infty} \frac{(n-1)B_{n,x(n-1)}}{\sum_{p=1}^{2n-2} B_{n,p}}
  =\frac{e\,(6-x)}{2\,(4 e -2)}\,e^{-x/2}.
\end{equation}
\end{theorem}
\begin{proof}
We begin by noting that as $n$ is taken to infinity it is rare,
amongst all linear chord diagrams, for a short chord to be found
nested directly inside of another chord. Specifically, if
$i-1,i,i+1,i+2$ are the positions of consecutive vertices, and $i,i+1$
are joined by a short chord, then it is unlikely that $i-1,i+2$ will
also be joined by a chord. To see this consider the diagrams with
exactly one short chord on $2n-2$ vertices. If one then inserts an
additional short chord into one of these diagrams (to thus produce a
diagram on $2n$ vertices), then there are $2n-1$ possible positions
for it, only one of which will produce the nested configuration.

To leading order we may therefore consider a bubble to be constructed
by the existence of some short chords (at least one). We then consider
the vertices not participating in these short chords to constitute a
(shorter) diagram devoid of short chords. The number $Z_n$ of diagrams
on $2n$ vertices devoid of short chords is given asymptotically,
according to the Poisson distribution, by $Z_n\simeq (2n-1)!!
\,e^{-1}$, and thus $Z_n \simeq (2n - 1) Z_{n-1}$.

Starting with a diagram devoid of short chords on $2(n-k)$ vertices,
where $k$ is order 1, we add $k$ short chords to produce a diagram on
$2n$ vertices. We may use one of these short chords (in conjunction
with the end of the diagram), or two of them, to bound a bubble of
size $p$ (which we take to be order $n$), see Figure
\ref{Figbubbasymp}. In the former case there are $k-1$
indistinguishable short chords whose positions must be summed over. In
the latter, the complex consisting of the bubble and its two bounding
chords must have its position summed over, in addition to the
remaining $k-2$ indistinguishable chords. We therefore have the
following asymptotic count for the number of bubbles of size $p$:

\begin{align}\nonumber
  B_{n,p}&\simeq\sum_k Z_{n-k} \left( 2\sum_{i_1<\cdots<i_{k-1}=1}^{2(n-k)-p-1}
1+ \sum_l\sum_{i_2<\cdots<i_{l-1}<i_1 <i_l\cdots<i_{k-1}=1}^{2(n-k)-p-1}1
\right)\\\nonumber
  &\simeq \frac{Z_n}{2n-1}\sum_k\frac{1}{(2n-{\cal O}(1))^{k-1}}\biggl(
  2\frac{(2n-p-{\cal
      O}(1))^{k-1}}{(k-1)!}
  +\frac{(2n-p-{\cal O}(1))^{k-1}}{(k-2)!}\biggr) \\\nonumber
  &\simeq \frac{Z_n}{2n-1}\sum_k\frac{k+1}{2^{k-1}(k-1)!}(2-x)^{k-1},
\end{align}
where we have used the relation $Z_n \simeq (2n - 1) Z_{n-1}$ once,
and then $k-1$ more times to express $Z_{n-k}$ in terms of $Z_n$. We
recall from Equation (\ref{asympmean}) that
\begin{equation}\nonumber
\sum_{p=1}^{2n-2} B_{n,p} \simeq (2n-1)!! \left(2 - e^{-1}\right),
\end{equation}
and using $Z_n\simeq (2n-1)!!\,e^{-1}$, we therefore find
\begin{equation}\nonumber
  \lim_{n\to\infty} \frac{(n-1)B_{n,x(n-1)}}{\sum_{p=1}^{2n-2} B_{n,p}}
  = \frac{1}{4e-2}\sum_k\frac{k+1}{2^{k-1}(k-1)!}(2-x)^{k-1}
= \frac{e\,(6-x)}{2\,(4 e -2)}\,e^{-x/2},
\end{equation}
where we have summed over $k$ from 1 to infinity.
\end{proof}
It is trivial to verify that the mean of $\rho(x)$ gives the result
from Equation (\ref{asympmean}). In Figure \ref{Figdist} we have
plotted the exact data for $n=150$ against $\rho(x)$.

\begin{figure}[H]
\begin{center}
 \includegraphics[width=4.in]{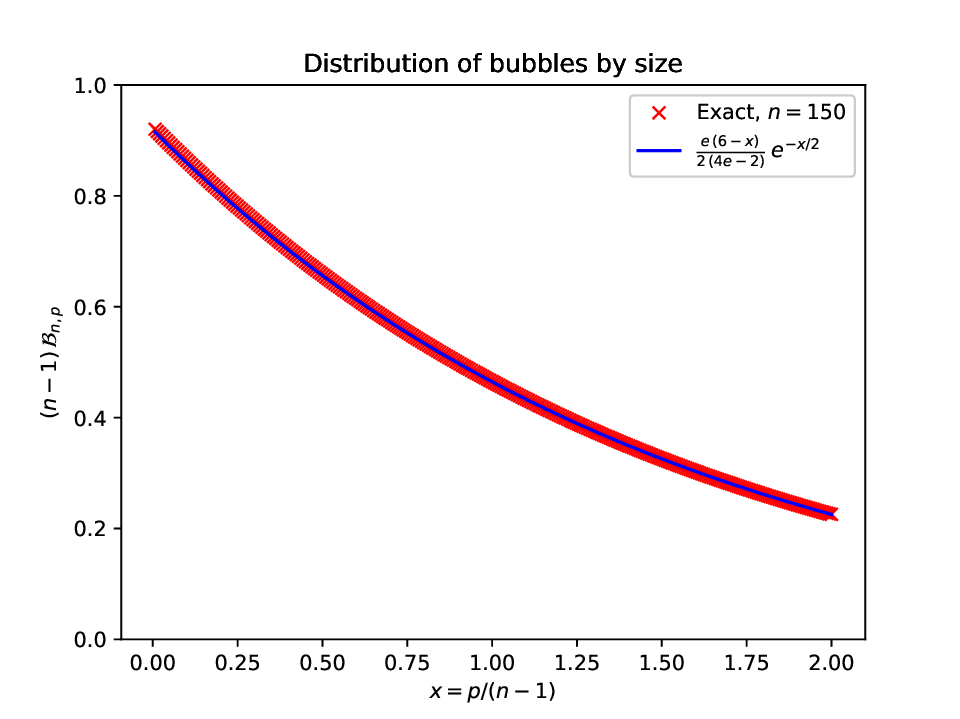}
\end{center}
\caption{A plot displaying $(n-1){\cal B}_{n,p} = (n-1)B_{n,p}/\sum_p
  B_{n,p}$ for $n=150$ and for $p=1$ to $2n-2$ in red ``X'''s, and the
  asymptotic distribution $\rho(x)$ as a blue solid curve.}
\label{Figdist}
\end{figure}

\section{Open questions and further research}

There are several directions in which the study of bubbles could be
extended. An immediate question is whether the counts in Lemmas
\ref{LemTot} and \ref{LemFrst} can be argued more directly from
combinatorial arguments. To give a sense of the difficulties present
here, it is instructive to consider the total number of bubbles,
i.e., the result of Lemma \ref{LemTot}. One way of estimating the
number of bubbles is to assume bubbles are bounded by single short
chords (and the ends of the diagram). This means we are discounting
configurations where short chords are adjacent to one another or found
at the ends of the diagram. By counting the possible positions of a
single short chord, and allowing the remaining $2n-2$ vertices to be
connected in every possible manner, i.e.,
\begin{equation}\nonumber
(2n-1) (2(n-1)-1)!!,
\end{equation}
one is counting the diagrams with $s$ short chords $s$ times, and
hence, according to the estimation, diagrams with $s+1$ bubbles $s$
times. We are thereby under-counting the number of bubbles by exactly
the number of linear chord diagrams, i.e., $(2n-1)!!$. Adding these
back we therefore find
\begin{equation}\nonumber
  \sum_{p=1}^{2n} B_{n,p} \simeq (2n-1) (2(n-1)-1)!! + (2n-1)!! =
  \frac{(2n-2)!(4n-2)}{2^{n-1}(n-1)!}.
\end{equation}
This is a $1+{\cal O}(1/n)$ multiplicative correction away from the
exact result given in Lemma \ref{LemTot}. The inclusion of the
configurations where short chords are adjacent seems unwieldy, and the
author has not found an elegant combinatorial argument to account for
them.

Another direction of future research is to refine the counting of
bubbles by the number of {\it internal} chords, i.e., chords with both
their endpoints found within the bubble. The author \cite{DY4} hopes
to report his findings with respect to this problem shortly. Those
chords which are not internal must bridge two bubbles; the
connectivity of a given bubble to others in the same diagram becomes a
natural further question.

Finally, the concept of a short chord extends readily to graphs other
than the path of length $2n$, cf.\ Young \cite{DY1}\cite{DY2}, and
it would be interesting to count bubbles on two (or higher)
dimensional grids or other more general graphs.

\section{Acknowledgment}

The author thanks the anonymous referee for their helpful comments and
suggestions for improving the paper.

\bigskip
\hrule
\bigskip

\noindent {\it 2010 Mathematics Subject Classification:} 
Primary 05A15; Secondary 05C70, 60C05. 

\noindent \emph{Keywords:} 
chord diagram, perfect matching. 
\bigskip
\hrule
\bigskip

\noindent (Concerned with sequences 
\seqnum{A278990}, \seqnum{A079267}, \seqnum{A367000}.)

\end{document}